\documentclass[reqno]{amsart}
\usepackage{amssymb}
\usepackage{hyperref}

\newtheorem{theorem}{Theorem}[section]
\newtheorem{lemma}[theorem]{Lemma}
\newtheorem{proposition}[theorem]{Proposition}
\newtheorem{corollary}[theorem]{Corollary}
\theoremstyle{definition}
\newtheorem{definition}[theorem]{Definition}
\newtheorem{example}[theorem]{Example}

\theoremstyle{remark}
\newtheorem{remark}[theorem]{Remark}

\numberwithin{equation}{section}      % onecolumn (second format)
\usepackage{graphicx}
\begin{document}

%\title{Insert your title here%\thanks{Grants or other notes


%\author{First Author         \and
       %Second Author %etc.
%}


%\institute{F. Author \at
%              first address \\
%              Tel.: +123-45-678910\\
%              Fax: +123-45-678910\\
%              \email{fauthor@example.com}           %  \\
%%             \emph{Present address:} of F. Author  %  if needed

\date{Received: date / Accepted: date}

\title{On recurrent sets of operators}
\author{ Mohamed Amouch and Otmane Benchiheb}

\address{\textsc{Mohamed Amouch and Otmane Benchiheb},
University Chouaib Doukkali.
Department of Mathematics, Faculty of science
Eljadida, Morocco}
\email{amouch.m@ucd.ac.ma}
\email{otmane.benchiheb@gmail.com}

\begin{abstract}
An operator $T$ acting on a Banach space $X$ is said to be recurrent if for each $U$; a nonempty open subset
of $X$, there exists $n\in\mathbb{N}$ such that $T^n(U)\cap U\neq\emptyset.$ In the present work, we generalize this notion from a single operator to a set $\Gamma$ of operators. As application, we study the recurrence of $C$-regularized group of operators.
\end{abstract}

%-----------------------------------

\maketitle

\section{Introduction and Preliminary}
%%%%%%%%%%%%%%%%%%%%%%%%%%%%%%%%%%%%%%%%%%%%%%%%%%%%%%%%%%%%%%%%%%%%%%%%%%%%%%%%%%%%%%%%%%%%%%%%%%%%%%%%%%
Let $X$ be a complex Banach space and $\mathcal{B}(X)$ the algebra of all bounded linear operators on $X$. By an operator, we always mean a bounded linear operator.
%%%%%%%%%%%%%%%%%%%%%%%%%%%%%%%%%%%%%%%%%%%%%%%%%%%%%%%%%%%%%%%%%%%%%%%%%%%%%%%%%%%%%%%%%%%%%%%%%%%%%%%%%%

The most studied notion in linear dynamics is that of hypercyclicity, supercyclicity and cyclicity: An
operator $T$ acting on a separable Banach space is said to be hypercyclic if there exists some vector $x\in X$
whose orbit under $T$;
$$ Orb(T,x):=\{T^nx \mbox{ : }n\geq0\}, $$
 is a dense subset of $X$. In this case, the vector $x$ is called a hypercyclic vector for $T$ and the set of all hypercyclic vectors of $T$ is denoted by $HC(T)$. 
% The first example of hypercyclic operator was given by Rolewicz in \cite{Rolewicz}. He
%proved that if $B$ is a backward shift on the Banach space $\ell^p(\mathbb{N})$, then $\lambda B$ is hypercyclic for any complex number $\lambda$ such that $\vert \lambda\vert > 1$.

The notion of supercyclic operators was introduced by Hilden and Wallen in \cite{HW}. An operator $T$ acting on $X$ is said to be supercyclic if the homogeneous orbit;
$$\mathbb{C}Orb(T,x):=\{\lambda T^nx \mbox{ : }\lambda\in\mathbb{C}\mbox{, } n\geq0\}, $$ 
is dense in $X$. In this case, the vector $x$ is called a supercyclic vector for $T$ and the set of all supercyclic vectors of $T$ is denoted by $SC(T)$.

%An operator $T$ is said to be cyclic if there is a vector $x \in X$, also called cyclic, such that
%$$ \mbox{span}\{Orb(T,x)\}=\mbox{span}\{T^nx \mbox{ : }n\geq0\}$$
%is a dense subset of $X$. We denote by $C(T)$ the set of all cyclic vectors of $T$. An example of a cyclic operator is the forward shift.
%Let $(e_n)_{n\geq0}$ be the canonical basis of $\ell^2(\mathbb{N})$. The forward shift S is defined by
%$$ Se_n=e_{n+1} \hspace{0.3cm}n\geq1.$$
%It is clear that $e_1$ is a cyclic vector for $S$.

For the more detailed information on hypercyclicity and supercyclicity,
the reader may refer to \cite{Bayart,Peris}.

  On the other hand, a very central notion
in topological dynamics is that of recurrence. This notion goes back to Poincar\'e and
Birkhoff and it refers to the existence of points in the space for which parts of their
orbits under a continuous map ''return'' to themselves. 

%%%%%%%%%%%%%%%%%%%%%%%%%%%%%%%%%%%%%%%%%%%%%%%%%%%%%%%%%%%%%%%%%%%%%%%%%%%%%%%%%%%%%%%%%%%%%%%%%%%%%%%%%%
A vector $x \in X $ is called recurrent for an operator $T\in\mathcal{B}(X)$ or $T$-recurrent vector if there exists a strictly increasing sequence of positive integers $(k_n)_{n\in\mathbb{N}}$ such that
$$ T^{k_n}x\longrightarrow x$$
as $n\longrightarrow\infty$. The set of all recurrent vectors for $T$ is denoted by $Rec(T)$.
%We have 
%$$ Rec(T)=\bigcap_{s=1}^{\infty}\bigcup_{n=0}^{\infty}\left\lbrace x\in X\mbox{ : }\left\|T^nx-x\right\|<\frac{1}{s}\right\rbrace,  $$
%which shows that the set of $T$-recurrent vectors is a $G_\delta$ set.
The operator $T$ itself is called recurrent if for each nonempty open set $U$ of $X$, there exists some $n\in\mathbb{N}$ such that 
$$ T^{-n}(U)\cap U\neq\emptyset. $$ 
%One can ask about the relation among the notion of recurrent vectors and the notion of recurrent operators :
%
%
%Let $T$ : $X \longrightarrow X$ be a bounded linear operator acting on a Banach space $X$. The following are equivalent$:$ 
%\begin{itemize}
%\item[$(1)$] $T$ is recurrent;
%\item[$(2)$] $\overline{Rec(T)}=X.$
%\end{itemize}
%
%As a consequence, if $T$ is recurrent, then the set of recurrent vectors for $T$ is a dense $G_\delta$ subset of $X$. In Example \ref{cex}, we will show that may an operator $T$ admits recurrent vectors without being recurrent operator.

For more information about recurrent vectors and recurrent operators, the reader may refer to \cite{CMP,CP,Furstenberg,KBA}.

%%%%%%%%%%%%%%%%%%%%%%%%%%%%%%%%%%%%%%%%%%%%%%%%%%%%%%%%%%%%%%%%%%%%%%%%%%%%%%%%%%%%%%%%%%%%%%%%%%%%%%%%%%

%Let $X$ and $Y$ be complex Banach spaces. If $T\in\mathcal{B}(X)$ and $S\in\mathcal{B}(Y)$, then $T$ and $S$ are called quasi-conjugate or quasi-similar if there exists an operator $\phi$ : $X\longrightarrow Y$ with dense range such that $S\circ\phi=\phi\circ T.$ If $\phi$ can be chosen to be a homeomorphism, then $T$ and $S$ are called conjugate or similar, see \cite[Definition 1.5]{Peris}.
%A property $\mathcal{P}$ is said to be preserved under quasi-similarity if for all $T\in\mathcal{B}(X)$, if $T$ has property $\mathcal{P}$, then every operator $S\in\mathcal{B}(Y)$ that is quasi-similar to $T$ has also property $\mathcal{P}$, see \cite[Definition 1.7]{Peris}. 
%%%%%%%%%%%%%%%%%%%%%%%%%%%%%%%%%%%%%%%%%%%%%%%%%%%%%%%%%%%%%%%%%%%%%%%%%%%%%%%%%%%%%%%%%%%%%%%%%%%%%%%%%%

%%%%%%%%%%%%%%%%%%%%%%%%%%%%%%%%%%%%%%%%%%%%%%%%%%%%%%%%%%%%%%%%%%%%%%%%%%%%%%%%%%%%%%%%%%%%%%%%%%%%%%%%%%

Recently, the notion of hypercyclic operators, supercyclic operators and cyclic operators was generalized to subset $\Gamma$ of $\mathcal{B}(X)$. A set $\Gamma$ of operators is called hypercyclic if there exists a vector $x$ in $X$ such that its orbit under $\Gamma$; 
 $$Orb(\Gamma,x)=\{Tx \mbox{ : }T\in\Gamma\}$$
  is dense in $X$, see \cite{lindynsetope,AKH}. If there exists a vector $x$ such that the homogeneous orbit;
$$\mathbb{C}Orb(\Gamma,x)=\{\alpha Tx\mbox{ : }T\in\Gamma\mbox{, }\alpha\in\mathbb{C}\},$$
 is dense in $X$ for some vector $x$, then $\Gamma$ is called a supercyclic set of operators, see \cite{somversupsetope,AKH1}. If 
$$\mbox{span}\{Orb(\Gamma,x)\}=\mbox{span}\{Tx \mbox{ : }T\in\Gamma\}$$
 is dense in $X$ for some vector $x$, then $\Gamma$ is cyclic, see \cite{cycsetope,AKH1}. 
 In each case,
such a vector x is called a hypercyclic vector, a supercyclic vector, a cyclic vector for $\Gamma$,
respectively.
%%%%%%%%%%%%%%%%%%%%%%%%%%%%%%%%%%%%%%%%%%%%%%%%%%%%%%%%%%%%%%%%%%%%%%%%%%%%%%%%%%%%%%%%%%%%%%%%%%%%%%%%%%

%%%%%%%%%%%%%%%%%%%%%%%%%%%%%%%%%%%%%%%%%%%%%%%%%%%%%%%%%%%%%%%%%%%%%%%%%%%%%%%%%%%%%%%%%%%%%%%%%%%%%%%%%%
In this paper, we introduce and study concepts of recurrent vectors and recurrent sets in a complex Banach space $X$. 

 In section 2, we introduce and study the notion of recurrent vector for a set $\Gamma$ of operators. We give some examples and we prove that the set of all recurrent vectors for a set $\Gamma$ is a $G_\delta$ type.

 In section 3, we introduce the notion of recurrent sets of operators. We prove that a set is recurrent if and only if the set of all recurrent vectors is dense. Moreover, we give a counterexample show that may a set $\Gamma$ admits recurrent vectors without being recurrent.
%%%%%%%%%%%%%%%%%%%%%%%%%%%%%%%%%%%%%%%%%%%%%%%%%%%%%%%%%%%%%%%%%%%%%%%%%%%%%%%%%%%%%%%%%%%%%%%%%%%%%%%%%% 
 
 In section 4, we give applications for $C$- regularized of operators. We show that the recurrent  $C$- regularized of operators exists on each complex Banach space $X$ and we prove by giving examples that some proprieties known for hypercyclic strongly continuous semigroups and $C$- regularized group of operators does not holds in general in the case of and $C$- regularized group of operators.
%%%%%%%%%%%%%%%%%%%%%%%%%%%%%%%%%%%%%%%%%%%%%%%%%%%%%%%%%%%%%%%%%%%%%%%%%%%%%%%%%%%%%%%%%%%%%%%%%%%%%%%%%
%%%%%%%%%%%%%%%%%%%%%%%%%%%%%%%%%%%%%%%%%%%%%%%%%%%%%%%%%%%%%%%%%%%%%%%%%%%%%%%%%%%%%%%%%%%%%%%%%%%%%%%%%%
\section{Recurrent Vectors of Sets of Operators}
\begin{definition}\label{def1}
A vector $x\in X\setminus\{0\}$ is said to be recurrent for $\Gamma$ or $\Gamma$-recurrent if there exist a sequence $\{k\}$ of positive integers and a sequence $\{T_{k}\} \subset\Gamma$ such that
$$T_{k}x\longrightarrow x $$
as $k\longrightarrow +\infty$. We denote by $Rec(\Gamma)$ the set of all recurrent vectors for $\Gamma.$
\end{definition}
%%%%%%%%%%%%%%%%%%%%%%%%%%%%%%%%%%%%%%%%%%%%%%%%%%%%%%%%%%%%%%%%%%%%%%%%%%%%%%%%%%%%%%%%%%%%%%%%%%%%%%%%%%
\begin{remark}
Let $X$ be a complex Banach space and $T\in\mathcal{B}(X)$. A vector $x\in X$ is recurrent for $T$ if and only if $x$ is a recurrent vector for the set
$$\Gamma:=\{T^n x\mbox{ : }n\geq0\}.$$
In this case, we write $Rec(T)$ instead of $Rec(\Gamma)$, see \cite{CMP}.
\end{remark}
Let $X$ be a complex Banach space and $\Gamma$ a subset of $\mathcal{B}(X)$. It is clear that if $x\in X$ is a hypercyclic vector for $\Gamma$, then it is a recurrent vector for $\Gamma$. The converse does not holds in general as shows the next example.
%%%%%%%%%%%%%%%%%%%%%%%%%%%%%%%%%%%%%%%%%%%%%%%%%%%%%%%%%%%%%%%%%%%%%%%%%%%%%%%%%%%%%%%%%%%%%%%%%%%%%%%%%%
\begin{example}\label{ex2}
Let $X$ be a complex Banach space and $(a_n)_{n\geq0}$ a sequence of complex numbers such that $a_n\longrightarrow 1$
as $n\longrightarrow +\infty$.
 For all $n\in\mathbb{N}$, let $T_n$ be an operator defined by$:$
$$\begin{array}{ccccc}
T_n & : & X & \longrightarrow & X \\
 & & x & \longmapsto & a_n x. \\
\end{array}$$
Let $\Gamma=\{T_n\mbox{ : }n\in\mathbb{N}\}\subset\mathcal{B}(X)$. For all $x\in X\setminus\{0\}$, we have
$$ T_nx=a_nx \longrightarrow x$$
as $n\longrightarrow +\infty$. This means that $x$ is a recurrent vector for $\Gamma$ for all $x\in X\setminus\{0\}$. However, $\Gamma$ admits no hypercyclic vectors.
\end{example}
%%%%%%%%%%%%%%%%%%%%%%%%%%%%%%%%%%%%%%%%%%%%%%%%%%%%%%%%%%%%%%%%%%%%%%%%%%%%%%%%%%%%%%%%%%%%%%%%%%%%%%%%%%
\begin{remark}
Let $X$ be a complex Banach space. We can deduce, from Example \ref{ex2}, that in each complex Banach space, with finite or infinite dimension, there exists a subset $\Gamma$ of $\mathcal{B}(X)$ which admits recurrent vectors.
\end{remark}
%%%%%%%%%%%%%%%%%%%%%%%%%%%%%%%%%%%%%%%%%%%%%%%%%%%%%%%%%%%%%%%%%%%%%%%%%%%%%%%%%%%%%%%%%%%%%%%%%%%%%%%%%%
%%%%%%%%%%%%%%%%%%%%%%%%%%%%%%%%%%%%%%%%%%%%%%%%%%%%%%%%%%%%%%%%%%%%%%%%%%%%%%%%%%%%%%%%%%%%%%%%%%%%%%%%%%
Let $X$ be a complex Banach space and  $\Gamma$ a subset of $ \mathcal{B}(X).$ We denote by $\{\Gamma\}^{'}$ the set of all elements of $\mathcal{B}(X)$ which commute with every element of $\Gamma.$ That is
$$ \{\Gamma\}^{'}:=\{S\in\mathcal{B}(X) \mbox{ : }TS=ST \mbox{ for all }T\in\Gamma\}. $$
%%%%%%%%%%%%%%%%%%%%%%%%%%%%%%%%%%%%%%%%%%%%%%%%%%%%%%%%%%%%%%%%%%%%%%%%%%%%%%%%%%%%%%%%%%%%%%%%%%%%%%%%%%
\begin{proposition}\label{p1}
Let $x\in Rec(\Gamma)$ and $S\in\mathcal{B}(X).$ If $S\in\{\Gamma\}^{'}$, then $Sx\in Rec(\Gamma)$.
\end{proposition}
%%%%%%%%%%%%%%%%%%%%%%%%%%%%%%%%%%%%%%%%%%%%%%%%%%%%%%%%%%%%%%%%%%%%%%%%%%%%%%%%%%%%%%%%%%%%%%%%%%%%%%%%%%
\begin{proof}
Since $x\in Rec(\Gamma)$, there exist a sequence $\{k\}$ of positive integers and a sequence $\{T_{k}\} \subset\Gamma$ such that
$$T_{k}x\longrightarrow x $$
as $k\longrightarrow +\infty$. Since $S$ is continuous and $S\in \{\Gamma\}^{'}$, it follows that
$$T_{k}(Sx)\longrightarrow Sx $$
as $k\longrightarrow +\infty$. Hence $Sx\in Rec(\Gamma).$
\end{proof}

%%%%%%%%%%%%%%%%%%%%%%%%%%%%%%%%%%%%%%%%%%%%%%%%%%%%%%%%%%%%%%%%%%%%%%%%%%%%%%%%%%%%%%%%%%%%%%%%%%%%%%%%%% 
\begin{corollary}
If $x\in Rec(\Gamma)$, then $\alpha x\in Rec(\Gamma)$, for all $\alpha \in \mathbb{C}\setminus\{0\}.$
\end{corollary}
%%%%%%%%%%%%%%%%%%%%%%%%%%%%%%%%%%%%%%%%%%%%%%%%%%%%%%%%%%%%%%%%%%%%%%%%%%%%%%%%%%%%%%%%%%%%%%%%%%%%%%%%%%
\begin{proof}
let $\alpha \in \mathbb{C}\setminus\{0\}$ and $x\in Rec(\Gamma)$, then $T=\alpha I$ is an operator such that $T\in\{\Gamma\}^{'}$. Hence, by Proposition \ref{p1}, $\alpha x\in Rec(\Gamma)$.
\end{proof}
%%%%%%%%%%%%%%%%%%%%%%%%%%%%%%%%%%%%%%%%%%%%%%%%%%%%%%%%%%%%%%%%%%%%%%%%%%%%%%%%%%%%%%%%%%%%%%%%%%%%%%%%%%

Let $X$ and $Y$ be complex Banach spaces. A set $\Gamma\subset\mathcal{B}(X)$ is said to be quasi-similar to a set $\Gamma_1\subset\mathcal{B}(Y)$ if there exists a continuous map $\phi$ : $X\longrightarrow Y$ of dense range and satisfies for all $T\in\Gamma,$ there exists some $S\in\Gamma_1$ such that $S\circ\phi=\phi\circ T$. If $\phi$ is a
homeomorphism, then $\Gamma$ and $\Gamma_1$ are called similar.
%%%%%%%%%%%%%%%%%%%%%%%%%%%%%%%%%%%%%%%%%%%%%%%%%%%%%%%%%%%%%%%%%%%%%%%%%%%%%%%%%%%%%%%%%%%%%%%%%%%%%%%%%%

The notion of recurrent vectors of a single operator is preserved under quasi-similarity, see \cite{CMP}.
In the following, we prove that the same result holds in the case of sets of operators.
\begin{proposition}\label{p3}
Assume that $\Gamma$ and $\Gamma_1$ are quasi-similar. If $x$ is a recurrent vector for $\Gamma$, then $\phi x$ is a recurrent vector for $\Gamma_1$. That is
$$\phi(Rec(\Gamma)\subset Rec(\Gamma_1).$$
\end{proposition}
%%%%%%%%%%%%%%%%%%%%%%%%%%%%%%%%%%%%%%%%%%%%%%%%%%%%%%%%%%%%%%%%%%%%%%%%%%%%%%%%%%%%%%%%%%%%%%%%%%%%%%%%%%
\begin{proof}
Since $\Gamma$ and $\Gamma_1$ are quasi-similar, there exists a continuous map $\phi$ : $X\longrightarrow Y$ with dense range such that for all $T\in\Gamma,$ there exists $S\in\Gamma_1$ satisfying $S\circ\phi=\phi\circ T$.
Assume that $\Gamma$ is recurrent in $X$, then there exist $x\in X\setminus\{0\}$, a sequence $\{k\}$ of positive integers and a sequence $\{T_{k}\} \subset\Gamma$ such that
$T_{k}x\longrightarrow x $
as $k\longrightarrow +\infty$. Hence, $\phi \circ T_{k}(x)\longrightarrow\phi x $ as $k\longrightarrow +\infty$. Foll all $k$, let $S_k\in\Gamma_1$ such that $S_k\circ\phi=T_k \circ\phi$. Hence, 
$$S_k(\phi x)\longrightarrow\phi x $$
 as $k\longrightarrow +\infty$, which implies that $\Gamma_1$ is recurrent in $Y$ and $\phi x\in Rec(\Gamma_1).$
\end{proof}
%%%%%%%%%%%%%%%%%%%%%%%%%%%%%%%%%%%%%%%%%%%%%%%%%%%%%%%%%%%%%%%%%%%%%%%%%%%%%%%%%%%%%%%%%%%%%%%%%%%%%%%%%%
\begin{remark}
In Proposition \ref{p3}, the condition of density of the range of $\phi$ is not necessary. Indeed, we need only that $\phi$ be continuous to prove this proposition.
\end{remark}
%%%%%%%%%%%%%%%%%%%%%%%%%%%%%%%%%%%%%%%%%%%%%%%%%%%%%%%%%%%%%%%%%%%%%%%%%%%%%%%%%%%%%%%%%%%%%%%%%%%%%%%%%%
\begin{corollary}
Let $X$ and $Y$ be complex Banach spaces, $\Gamma$ a subset of $\mathcal{B}(X)$ and $\Gamma_1$ a subset $\mathcal{B}(Y)$.
Assume that $\Gamma$ and $\Gamma_1$ are similar. A vector $x\in X$ is a recurrent vector for $\Gamma$ if and only if $\phi x$ is a recurrent vector for $\Gamma_1$. That is
$$\phi(Rec(\Gamma)=Rec(\Gamma_1).$$
\end{corollary}
%%%%%%%%%%%%%%%%%%%%%%%%%%%%%%%%%%%%%%%%%%%%%%%%%%%%%%%%%%%%%%%%%%%%%%%%%%%%%%%%%%%%%%%%%%%%%%%%%%%%%%%%%%
%%%%%%%%%%%%%%%%%%%%%%%%%%%%%%%%%%%%%%%%%%%%%%%%%%%%%%%%%%%%%%%%%%%%%%%%%%%%%%%%%%%%%%%%%%%%%%%%%%%%%%%%%%

%Let $\{X_i\}_{i=1}^{n}$ be a family of complex Banach spaces and let $\Gamma_i$ be a subset of $\mathcal{B}(X_i)$, for all $1\leq i\leq n$. Recall that $\oplus_{i=1}^n X_i$ is the space defined by
%$$\oplus_{i=1}^n X_i=X_1\times X_2\times\dots \times X_n=\{\oplus_{i=1}^n x_i=(x_1,x_2,\dots,x_n) \mbox{ : }x_i\in X_i\mbox{, }1\leq i\leq n\},$$
%and $\oplus_{i=1}^n\Gamma_i$ is the subset of $\mathcal{B}(\oplus_{i=1}^n X_i)$ defined by
%$$\oplus_{i=1}^n\Gamma_i=\{\oplus_{i=1}^n T_i=T_1\times T_2\times\dots \times T_n\mbox{ : }T_i\in\Gamma_i\mbox{, }1\leq i\leq n\},$$
%where $T_1\times T_2\times\dots \times T_n$ is the operator defined in $\oplus_{i=1}^n X_i$ by
%$$\begin{array}{ccccc}
%\oplus_{i=1}^nT_i & : & \oplus_{i=1}^nX_i & \longrightarrow & \oplus_{i=1}^nX_i \\
% & & \oplus_{i=1}^nx_i & \longmapsto & \oplus_{i=1}^nT_i x_i. \\
%\end{array}$$
%%%%%%%%%%%%%%%%%%%%%%%%%%%%%%%%%%%%%%%%%%%%%%%%%%%%%%%%%%%%%%%%%%%%%%%%%%%%%%%%%%%%%%%%%%%%%%%%%%%%%%%%%%
\begin{proposition}\label{p4}
Let $\{X_i\}_{i=1}^{n}$ be a family of complex Banach spaces and $\Gamma_i$ a subset of
$\mathcal{B}(X_i),$ for all $1\leq i\leq n$. If  $(x_1,x_2,\dots,x_n)\in Rec(\oplus_{i=1}^n\Gamma_i)$, then $x_i\in Rec(\Gamma_i)$, for all $1\leq i\leq n$. That is
$$ Rec(\oplus_{i=1}^n\Gamma_i)\subset \oplus_{i=1}^nRec(\Gamma_i). $$
\end{proposition}
%%%%%%%%%%%%%%%%%%%%%%%%%%%%%%%%%%%%%%%%%%%%%%%%%%%%%%%%%%%%%%%%%%%%%%%%%%%%%%%%%%%%%%%%%%%%%%%%%%%%%%%%%%
\begin{proof}
Let $(x_1,x_2,\dots,x_n)\in Rec(\oplus_{i=1}^n\Gamma_i)$, then there exist a sequence $\{k\}$ of positive integers and a sequence $\{(T^1_{k},T^2_{k},\dots,T^n_{k})\} \subset \oplus_{i=1}^n\Gamma_i$ such that
$$(T^1_{k}x_1,T^2_{k}x_2,\dots,T^n_{k}x_n)=(T^1_{k},T^2_{k},\dots,T^n_{k})(x_1,x_2,\dots,x_n)\longrightarrow (x_1,x_2,\dots,x_n) $$
as $k\longrightarrow +\infty$. Hence, for all $1\leq i\leq n,$ we have $T^i_{k}x_i\longrightarrow x_i$ as $k\longrightarrow +\infty$. Thus, for all $1\leq i\leq n,$ $\Gamma_i$ is recurrent in $X_i$ and $x_i\in Rec(\Gamma_i).$
\end{proof}
%%%%%%%%%%%%%%%%%%%%%%%%%%%%%%%%%%%%%%%%%%%%%%%%%%%%%%%%%%%%%%%%%%%%%%%%%%%%%%%%%%%%%%%%%%%%%%%%%%%%%%%%%%

A characterization of the set of recurrent vectors of set of operators on a complex Banach space is due the next proposition.
%%%%%%%%%%%%%%%%%%%%%%%%%%%%%%%%%%%%%%%%%%%%%%%%%%%%%%%%%%%%%%%%%%%%%%%%%%%%%%%%%%%%%%%%%%%%%%%%%%%%%%%%%%

\begin{proposition}\label{p2}
 The set of all recurrent vectors for a set $\Gamma$ is even empty or a $G_\delta$ type. In the last case we have
$$Rec(\Gamma)=\bigcap_{n\geq1}\bigcup_{T\in \Gamma}\left\lbrace x\in X \mbox{ : }\Vert Tx -x \Vert<\frac{1}{n} \right\rbrace .$$
\end{proposition}
%%%%%%%%%%%%%%%%%%%%%%%%%%%%%%%%%%%%%%%%%%%%%%%%%%%%%%%%%%%%%%%%%%%%%%%%%%%%%%%%%%%%%%%%%%%%%%%%%%%%%%%%%%
\begin{proof}
Let $x\in Rec(\Gamma)$. Then there exist a sequence $\{k\}$ of positive integers and a sequence $\{T_{k}\} \subset\Gamma$ such that
$T_{k}x\longrightarrow x $
as $k\longrightarrow +\infty$. Hence, for all $n\geq1$, there exists $k$ such that $\Vert T_kx-x\Vert<\frac{1}{n}$, this implies that $\displaystyle x\in \bigcap_{n\geq1}\bigcup_{T\in \Gamma}\left\lbrace x\in X \mbox{ : }\Vert Tx -x \Vert<\frac{1}{n} \right\rbrace$. For the converse, let  $\displaystyle x\in \bigcap_{n\geq1}\bigcup_{T\in \Gamma}\left\lbrace x\in X \mbox{ : }\Vert Tx -x \Vert<\frac{1}{n} \right\rbrace$, then for all $n\geq1 $ there exists $T_n\in\Gamma$ such that $\Vert Tx -x \Vert<\frac{1}{n}$, this means that $T_{n}x\longrightarrow x $ as $n\longrightarrow +\infty$. Hence, $x\in Rec(\Gamma).$ Since for all $n\geq1$ the set $\left\lbrace x\in X \mbox{ : }\Vert Tx -x \Vert<\frac{1}{n} \right\rbrace$ is open, it follows that $Rec(\Gamma)$ is a $G_\delta$ type.
\end{proof}
%%%%%%%%%%%%%%%%%%%%%%%%%%%%%%%%%%%%%%%%%%%%%%%%%%%%%%%%%%%%%%%%%%%%%%%%%%%%%%%%%%%%%%%%%%%%%%%%%%%%%%%%%%
\begin{theorem}\label{t2} 
Assume that for all $T$, $S\in\Gamma$ we have $TS\in\Gamma$ and $(\lambda_T)_{T\in\Gamma}$ be a (multiplicative) semi-group inside $\mathbb{T}$. Then $x\in Rec(\Gamma)$ if and only if $x\in Rec(\Gamma_1)$, where 
$$\Gamma_1:=\{\lambda_T T \mbox{ : }T\in\Gamma\}$$
 and $\mathbb{T}=\{\alpha\in\mathbb{C} \mbox{ : }\vert\alpha\vert=1\}$ is the unite circle.
\end{theorem}
%%%%%%%%%%%%%%%%%%%%%%%%%%%%%%%%%%%%%%%%%%%%%%%%%%%%%%%%%%%%%%%%%%%%%%%%%%%%%%%%%%%%%%%%%%%%%%%%%%%%%%%%%%
\begin{proof}
To prove the proposition we only need to prove that $ Rec(\Gamma)\subset Rec(\Gamma_1)$. Let $x$ be a recurrent vector for $\Gamma)$. We define a subset $F$ of $\mathbb{T}$ by
$$ F:=\{\mu\in\mathbb{T}\mbox{ : }(\lambda_k T_k)x\longrightarrow\mu x \mbox{ for some sequence }(k)\subset\mathbb{N} \mbox{ with }k\longrightarrow\infty\}. $$

To show that $x\in Rec(\Gamma)$ we need to prove that $1\in F$. For that, we begin by proving that $F\neq\emptyset$. 

Since $x$ is a recurrent vector for $\Gamma$, there exists a 
sequence of positive integers $(k)\subset\mathbb{N}$ such that 
$$T_kx\longrightarrow x.$$
Without loss of generality we my suppose that $\lambda_k \longrightarrow\rho$ for
some $\rho\in\mathbb{T}$. We conclude that $(\lambda_k T_k)x\longrightarrow\rho x$. This means that $\rho \in F$.

Now, let
$\mu_1$, $\mu_2 \in F$ and let $\varepsilon > 0$ fixed. Since $\mu_1 \in F$ there exists $n_1\in\mathbb{N}$ such that
$$\Vert \lambda_{n_1}T_{n_1}x-\mu_1 x\Vert<\frac{\varepsilon}{2}.$$
Since $\mu_2 \in F$ there exists $n_2\in\mathbb{N}$ such that
$$\Vert \lambda_{n_2}T_{n_2}x-\mu_2 x\Vert<\frac{\varepsilon}{2\Vert\lambda_{n_1}T_{n_1}\Vert}.$$

We thus get
\begin{align*}
\Vert (\lambda_{n_1}\lambda_{n_2}T_{n_1}T_{n_2})x-\mu_1\mu_2x\Vert  &\leq \Vert (\lambda_{n_1} T_{n_1})(\lambda_{n_2} T_{n_2}x-\mu_2x)\Vert+\Vert \mu_2(\lambda_{n_1} T_{n_1}x-\mu_1x)\\
                                                                    &\leq \Vert \lambda_{n_1} T_{n_1}\Vert \Vert \lambda_{n_2} T_{n_2}x-\mu_2 x\Vert+\frac{\varepsilon}{2}<\varepsilon.
\end{align*}
So $\mu_1 \mu_2\in F.$ Hence $F$ is a (multiplicative) semi-group inside $\mathbb{T}.$

Let $\rho \in F$, then  for all $n\in\mathbb{N}$, we have $\rho^n \in F$. We have two cases

\begin{itemize}
\item If $\rho$ is a rational rotation this means that $1 \in F$ and we are done.
\item If not, there exists
sequence of positive integers $\tau_k$ such that $\rho^{\tau_k} \longrightarrow 1$. 
Since $F$ is closed, it follows that that $1 \in F$.
\end{itemize}
\end{proof}
%%%%%%%%%%%%%%%%%%%%%%%%%%%%%%%%%%%%%%%%%%%%%%%%%%%%%%%%%%%%%%%%%%%%%%%%%%%%%%%%%%%%%%%%%%%%%%%%%%%%%%%%%%
%%%%%%%%%%%%%%%%%%%%%%%%%%%%%%%%%%%%%%%%%%%%%%%%%%%%%%%%%%%%%%%%%%%%%%%%%%%%%%%%%%%%%%%%%%%%%%%%%%%%%%%%%%
%%%%%%%%%%%%%%%%%%%%%%%%%%%%%%%%%%%%%%%%%%%%%%%%%%%%%%%%%%%%%%%%%%%%%%%%%%%%%%%%%%%%%%%%%%%%%%%%%%%%%%%%%%%%%%%%%%%%%%%%%%%%%%%%%%%%%%%%%%%%%%%%%%%%%%%%%%%%%%%%%%%%%%%%%%%%%%%%%%%%%%%%%%%%%%%%%%%%%%%%%%%%%%%%%%%%%%%%%%%%%%%%%%%%%%%%%%%%%%%%%%%%%%%%%%%%%%%%%%%%%%%%%%%%%%%%%%%%%%%%%%%%%%%%%%%%%%%%%%%%%%%%%%%%%%%%%%%%%%
\section{Recurrent Sets of Operators}

In the following definition, we introduce the notion of recurrence of sets of operators which generalizes the notion of recurrence of a single operator.  
%%%%%%%%%%%%%%%%%%%%%%%%%%%%%%%%%%%%%%%%%%%%%%%%%%%%%%%%%%%%%%%%%%%%%%%%%%%%%%%%%%%%%%%%%%%%%%%%%%%%%%%%%%
\begin{definition}
A set $\Gamma\subset$ is called recurrent if for each
nonempty open subset $U$ of $X$ there exists some operator $T\in\Gamma$ such that
$$T(U)\cap U\neq\emptyset.$$
\end{definition} 
%%%%%%%%%%%%%%%%%%%%%%%%%%%%%%%%%%%%%%%%%%%%%%%%%%%%%%%%%%%%%%%%%%%%%%%%%%%%%%%%%%%%%%%%%%%%%%%%%%%%%%%%%%
\begin{remark}
Let $X$ be a complex Banach space.
An operator $T\in\mathcal{B}(X)$ is recurrent as an operator if and only if the set
$$\Gamma:=\{T^n \mbox{ : }n\geq0\}$$
 is recurrent as a set of operators.
\end{remark}
%%%%%%%%%%%%%%%%%%%%%%%%%%%%%%%%%%%%%%%%%%%%%%%%%%%%%%%%%%%%%%%%%%%%%%%%%%%%%%%%%%%%%%%%%%%%%%%%%%%%%%%%%%

The recurrence of a single operators is preserved under quasi-similarity, see \cite{CMP}. The following proposition proves that the same result holds in the case of sets of operators.
%%%%%%%%%%%%%%%%%%%%%%%%%%%%%%%%%%%%%%%%%%%%%%%%%%%%%%%%%%%%%%%%%%%%%%%%%%%%%%%%%%%%%%%%%%%%%%%%%%%%%%%%%%

%%%%%%%%%%%%%%%%%%%%%%%%%%%%%%%%%%%%%%%%%%%%%%%%%%%%%%%%%%%%%%%%%%%%%%%%%%%%%%%%%%%%%%%%%%%%%%%%%%%%%%%%%%
\begin{proposition}\label{prop1}
Assume that $\Gamma\subset\mathcal{B}(X)$ and $\Gamma_1\subset\mathcal{B}(Y)$ are quasi-similar.
 If $\Gamma$ is a recurrent set in $X$, then $\Gamma_1$ is a recurrent set in $Y$.
\end{proposition}
%%%%%%%%%%%%%%%%%%%%%%%%%%%%%%%%%%%%%%%%%%%%%%%%%%%%%%%%%%%%%%%%%%%%%%%%%%%%%%%%%%%%%%%%%%%%%%%%%%%%%%%%%%
\begin{proof}
Since  $\Gamma$ and $\Gamma_1$ are quasi-similar, there exists a continuous map $\phi$ : $X\longrightarrow Y$ with dense range such that for all $T\in\Gamma,$ there exists $S\in\Gamma_1$ satisfying $S\circ\phi=\phi\circ T$.
Let $U$ be a nonempty open subset of $Y$. Since $\phi$ is of dense range, $\phi^{-1}(U)$ is nonempty and open. If $\Gamma$ is recurrent in $X$, then there exist $y\in \phi^{-1}(U)$ and $T\in\Gamma$ such that $Ty\in\phi^{-1}(U)$, which implies that $\phi(y)\in U$ and $\phi(Ty)\in U$. Let $S\in\Gamma_1$ such that $S\circ\phi=\phi\circ T$, then $\phi(y)\in U$ and $S\phi(y)\in U$. From this, we deduce that $\Gamma_1$ is recurrent in $Y$.
\end{proof}
%%%%%%%%%%%%%%%%%%%%%%%%%%%%%%%%%%%%%%%%%%%%%%%%%%%%%%%%%%%%%%%%%%%%%%%%%%%%%%%%%%%%%%%%%%%%%%%%%%%%%%%%%%
\begin{corollary}
Assume that $\Gamma\subset\mathcal{B}(X)$ and $\Gamma_1\subset\mathcal{B}(Y)$ are quasi-similar.
 Then, $\Gamma$ is recurrent in $Y$ if and only if $\Gamma_1$ is recurrent in $Y$.
\end{corollary}
%%%%%%%%%%%%%%%%%%%%%%%%%%%%%%%%%%%%%%%%%%%%%%%%%%%%%%%%%%%%%%%%%%%%%%%%%%%%%%%%%%%%%%%%%%%%%%%%%%%%%%%%%%
%%%%%%%%%%%%%%%%%%%%%%%%%%%%%%%%%%%%%%%%%%%%%%%%%%%%%%%%%%%%%%%%%%%%%%%%%%%%%%%%%%%%%%%%%%%%%%%%%%%%%%%%%%
%%%%%%%%%%%%%%%%%%%%%%%%%%%%%%%%%%%%%%%%%%%%%%%%%%%%%%%%%%%%%%%%%%%%%%%%%%%%%%%%%%%%%%%%%%%%%%%%%%%%%%%%%%
In the following result, we give necessary and sufficient conditions for a set of operators to be recurrent.
%%%%%%%%%%%%%%%%%%%%%%%%%%%%%%%%%%%%%%%%%%%%%%%%%%%%%%%%%%%%%%%%%%%%%%%%%%%%%%%%%%%%%%%%%%%%%%%%%%%%%%%%%%
\begin{theorem}\label{tt}
Let $\Gamma$ be a subset of $\mathcal{B}(X)$. The following assertions are equivalent$:$
\begin{itemize}
\item[$(i)$] $\Gamma$ is recurrent;
\item[$(ii)$] For each $x\in X,$ there exists sequences $\{x_k\}$ in $X$ and $\{T_k\}$ in $\Gamma$ such that
$$x_k\longrightarrow x\hspace{0.3cm}\mbox{ and }\hspace{0.3cm}T_k( x_k)\longrightarrow x;$$
\item[$(iii)$] For each $x\in X$ and for $W$ a neighborhood of $0$, there exist $z\in X$ and $T\in\Gamma$  such that
$$T( z)-x\in W \hspace{0.3cm}\mbox{ and }\hspace{0.3cm} x-z\in W. $$
\end{itemize}
\end{theorem}
%%%%%%%%%%%%%%%%%%%%%%%%%%%%%%%%%%%%%%%%%%%%%%%%%%%%%%%%%%%%%%%%%%%%%%%%%%%%%%%%%%%%%%%%%%%%%%%%%%%%%%%%%%
\begin{proof}$(i)\Rightarrow(ii)$
Let $x\in X$. For all $k\geq1$, let $U_k=B(x,\frac{1}{k})$. Then $U_k$ is a nonempty open subset of $X$. Since $\Gamma$ is recurrent, there exists $T_k\in\Gamma$ such that $T_k(U_k)\cap U_k\neq\emptyset$. For all $k\geq1$, let $x_k\in U_k$ such that $T_k( x_k)\in U_k$, then
$$\Vert x_k-x \Vert<\frac{1}{k}\hspace{0.3cm}\mbox{ and }\hspace{0.3cm}\Vert T_k(x_k)-x \Vert<\frac{1}{k}$$
which implies that
$$x_k\longrightarrow x\hspace{0.3cm}\mbox{ and }\hspace{0.3cm} T_k(x_k)\longrightarrow x.$$
%%%%%%%%%%%%%%%%%%%%%%%%%%%%%%%%%%%%%%%%%%%%%%%%%%%%%%%%%%%%%%%%%%%%%%%%%%%%%%%%%%%%%%%%%%%%%%%%%%%%%%%%%%

$(ii)\Rightarrow(iii)$ Let $x\in X$. There exists sequences $\{x_k\}$ in $X$ and $\{T_k\}$ in $\Gamma$ such that
$$x_k-x\longrightarrow 0\hspace{0.3cm}\mbox{ and }\hspace{0.3cm}T_k(x_k)-x\longrightarrow 0.$$
If $W$ is a neighborhood of $0$, then there exists $N\in\mathbb{N}$ such that $x-x_k\in W$ and $T_k(x_k)-x\in W$, for all $k\geq N$.
%%%%%%%%%%%%%%%%%%%%%%%%%%%%%%%%%%%%%%%%%%%%%%%%%%%%%%%%%%%%%%%%%%%%%%%%%%%%%%%%%%%%%%%%%%%%%%%%%%%%%%%%%%

$(iii)\Rightarrow(i)$ Let $U$ be a nonempty open subsets of $X$. Then there exists $x\in X$ such that $x\in U$. Since for all $k\geq1$,  $W_k=B(0,\frac{1}{k})$ is a neighborhood of $0$,  there exist $z_k\in X$ and $T_k\in\Gamma$ such that
$$\Vert T_k( z_k)-x\Vert<\frac{1}{k}\hspace{0.3cm}\mbox{ and }\hspace{0.3cm}\Vert x-z_k\Vert<\frac{1}{k}.$$
This implies that
$$z_k\longrightarrow x\hspace{0.3cm}\mbox{ and }\hspace{0.3cm}T_k(z_k)\longrightarrow x.$$
 Since $U$ is a nonempty open subset of $X$ and $x\in U$, there exists $N\in\mathbb{N}$ such that $z_k\in U$ and $T_k(z_k)\in U$, for all $k\geq N.$
\end{proof}
%%%%%%%%%%%%%%%%%%%%%%%%%%%%%%%%%%%%%%%%%%%%%%%%%%%%%%%%%%%%%%%%%%%%%%%%%%%%%%%%%%%%%%%%%%%%%%%%%%%%%%%%%%
\begin{proposition}
Let $\{X_i\}_{i=1}^{n}$ be a family of complex Banach spaces and $\Gamma_i$ a subset of
$\mathcal{B}(X_i),$ for all $1\leq i\leq n$. If  $\oplus_{i=1}^n(\Gamma_i)$ is recurrent in $\oplus_{i=1}^n(X_i)$, then $\Gamma_i$ is recurrent in $X_i$, for all $1\leq i\leq n$.
\end{proposition}
%%%%%%%%%%%%%%%%%%%%%%%%%%%%%%%%%%%%%%%%%%%%%%%%%%%%%%%%%%%%%%%%%%%%%%%%%%%%%%%%%%%%%%%%%%%%%%%%%%%%%%%%%%
\begin{proof}
Assume that $\oplus_{i=1}^n\Gamma_i$ is recurrent in $\oplus_{i=1}^n X_i$. If $U_i$ be a nonempty open set of $X_i$ for $1 \leq i\leq n$, then $U_1\times\dots\times U_n$ is a nonempty set of $\oplus_{i=1}^n(\Gamma_i)$. There exists $T_k^1\times\dots T_k^n$ such that 
$$ T_k^1\times\dots\times T_k^n(U_1\times\dots\times U_n)\cap U_1\times\dots\times U_n\neq\emptyset.$$
It follows that $ T_k^i(U_i)\cap U_i\neq\emptyset$ for all $1 \leq i\leq n$. Hence $\Gamma_i$ is recurrent in $X_i$ for all $1 \leq i\leq n$. 
\end{proof}

%%%%%%%%%%%%%%%%%%%%%%%%%%%%%%%%%%%%%%%%%%%%%%%%%%%%%%%%%%%%%%%%%%%%%%%%%%%%%%%%%%%%%%%%%%%%%%%%%%%%%%%%%%
The natural question here is about the relationship between the set $Rec(\Gamma)$ of all recurrent vectors for a set $\Gamma$ and the recurrence of $\Gamma$ itself. In the following theorem, we prove the equivalence between the recurrence of $\Gamma$ and the density of $Rec(\Gamma)$ in the space $X$.
%%%%%%%%%%%%%%%%%%%%%%%%%%%%%%%%%%%%%%%%%%%%%%%%%%%%%%%%%%%%%%%%%%%%%%%%%%%%%%%%%%%%%%%%%%%%%%%%%%%%%%%%%%
\begin{theorem}\label{t1}
Let $\Gamma\subset\mathcal{B}(X)$.
The following assertions are equivalent$:$
\begin{itemize}
\item[$(i)$] $Rec(\Gamma)$ is dense in $X$;
\item[$(ii)$] $\Gamma$ is recurrent.
\end{itemize}
\end{theorem}
%%%%%%%%%%%%%%%%%%%%%%%%%%%%%%%%%%%%%%%%%%%%%%%%%%%%%%%%%%%%%%%%%%%%%%%%%%%%%%%%%%%%%%%%%%%%%%%%%%%%%%%%%%
\begin{proof}
$(i)\Rightarrow(ii)$ : We suppose that $Rec(\Gamma)$ is a dense subset of $X$ and let $U$ be a nonempty open subset of $X$.
Then $Rec(\Gamma)\cap U\neq\emptyset$.
This implies that there exists $x\in X$ a recurrent vector for $\Gamma$ such that $x\in U.$
There exist a sequence $\{k\}$ of positive integers and a sequence $\{T_{k}\}$ of $\Gamma$ such that
$T_{k}x\longrightarrow x $
as $k\longrightarrow +\infty$.
Since $U$ is open and $x\in U$, there exists $N\in\mathbb{N}$ such that
$T_N(U)\cap U\neq\emptyset,$ and hence $\Gamma$ is recurrent.\\
$(ii)\Rightarrow(i)$ : Assume that $\Gamma$ is recurrent, we will prove that $Rec(\Gamma)$ is dense in $X$.
Let
$$B:=B(x,\varepsilon)$$
be a fixed open ball for some $x \in X$ and $\varepsilon < 1$. 
Since $\Gamma$ is recurrent, it follows that there
 exists an operator $T_1\in \Gamma$ such that $T_1(B)\cap B\neq\emptyset.$
 Hence, there exists $x_1\in X$ such that $x_1\in T_1^{-1}(B)\cap B$.
There exists $\varepsilon_1 < \frac{1}{2}$ such that
$$B_2 := B(x_1, \varepsilon_1)\subset B \cap T_1^{-1}(B)$$
since $T_1$ is continuous. 
Again, since $\Gamma$ is recurrent, there exists some $T_2\in\Gamma$ and some $x_2\in X$ such that $x_2 \in T_2^{-1}(B_2)\cap B_2$. 
Now we use the continuity of $T_2$ to conclude that there exists $\varepsilon_2<\frac{1}{2^2}$
 such that
 $$B_3 := B(x_2, \varepsilon_2)\subset B_2 \cap T_2^{-1}(B_2).$$
 Continuing inductively we construct a sequence $(x_k)_{k\in\mathbb{N}}$ of $ X$, a
sequence $(T_k)_{k\in\mathbb{N}}$ of $\Gamma$ and a sequence of positive real numbers $\varepsilon_k<\frac{1}{2^k}$, such that
$$ B(x_k, \varepsilon_k) \subset B(x_{k-1},\varepsilon_{k-1})\hspace{0.3cm}\mbox{ and }\hspace{0.3cm} T_k(B(x_k,\varepsilon_k))\subset B(x_{k-1},\varepsilon_{k-1}). $$
By hypothesis, $X$ is a Banach space, hence it is complete. We use Cantor's theorem to conclude that there exists some $y\in X$ such that
\begin{equation}\label{eqr}
\bigcap_{k}B(x_k,\varepsilon_k)=\{y\}. 
\end{equation} 
Since $y\in B$ the original ball, to finish the proof, it suffices to show that $y$ is recurrent vector for $\Gamma$. By \ref{eqr}, we have $y\in B(x_{k},\varepsilon_{k})$ for all $k$. This is equivalent to the fact that
\begin{equation}\label{equ1}
\Vert  x_{k}-y \Vert< \varepsilon_{k}.
\end{equation}
On the other hand, $T_{k+1}y\in B(x_{k},\varepsilon_{k+1})$ since $T_{k+1}(B(x_{k+1},\varepsilon_{k+1}))\subset B(x_{k},\varepsilon_{k})$, hence
\begin{equation}\label{equ2}
\Vert T_{k+1}y-x_k \Vert<\varepsilon_{k+1}. 
\end{equation}

Now using (\ref{equ1}) and (\ref{equ2}) we conclude that
$$ \Vert T_{k+1}y-y\Vert\leq \Vert  x_{k}-y \Vert+\Vert T_{k+1}y-x_k \Vert< \frac{1}{2^k}+\frac{1}{2^{k+1}}.$$

Hence $T_k y \longrightarrow y$ as $k\longrightarrow+\infty$, that is, $y$ is a recurrent point in the
original ball $B$ and the proof is completed.
\end{proof}
%%%%%%%%%%%%%%%%%%%%%%%%%%%%%%%%%%%%%%%%%%%%%%%%%%%%%%%%%%%%%%%%%%%%%%%%%%%%%%%%%%%%%%%%%%%%%%%%%%%%%%%%%%
\begin{remark}
Observe that the previous proposition remains valid whenever $\Gamma$ is a set of continuous map on a complete metric space $X.$
\end{remark}
%%%%%%%%%%%%%%%%%%%%%%%%%%%%%%%%%%%%%%%%%%%%%%%%%%%%%%%%%%%%%%%%%%%%%%%%%%%%%%%%%%%%%%%%%%%%%%%%%%%%%%%%%%

In particular, Theorem \ref{t1} proves that if $\Gamma$ is recurrent, then it admits a nontrivial recurrent vector. The following example shows that the converse does not holds in general even in the case of single operator.
%%%%%%%%%%%%%%%%%%%%%%%%%%%%%%%%%%%%%%%%%%%%%%%%%%%%%%%%%%%%%%%%%%%%%%%%%%%%%%%%%%%%%%%%%%%%%%%%%%%%%%%%%%

 Recall from \cite{CMP}, that if an operator acting on a complex Banach space is recurrent, then it is of dense range.
%%%%%%%%%%%%%%%%%%%%%%%%%%%%%%%%%%%%%%%%%%%%%%%%%%%%%%%%%%%%%%%%%%%%%%%%%%%%%%%%%%%%%%%%%%%%%%%%%%%%%%%%%%
\begin{example}\label{cex}
Let $X=\ell^2(\mathbb{N})$. Let $T$ be a linear operator defined in $\ell^2(\mathbb{N})$ by
$$ Te_1=e_1 \hspace{0.3cm} \mbox{ and }\hspace{0.3cm}Te_k=0 \mbox{ for all }k\geq 2.$$
Let $\Gamma=\{T^n \mbox{ : }n\geq0\}.$ Since $T^ne_1=e_1$, it follows that
$$ T^ne_1\longrightarrow e_1 $$
as $n\longrightarrow\infty.$
Hence, $e_1$ is a recurrent vector for $T$. 
Since $T$ is not of dense range, it follows that $T$ can not be recurrent.   
\end{example}
%%%%%%%%%%%%%%%%%%%%%%%%%%%%%%%%%%%%%%%%%%%%%%%%%%%%%%%%%%%%%%%%%%%%%%%%%%%%%%%%%%%%%%%%%%%%%%%%%%%%%%%%%

Using Theorem \ref{t2} and Theorem \ref{t1}, it easy to prove the next Proposition.
%%%%%%%%%%%%%%%%%%%%%%%%%%%%%%%%%%%%%%%%%%%%%%%%%%%%%%%%%%%%%%%%%%%%%%%%%%%%%%%%%%%%%%%%%%%%%%%%%%%%%%%%%%
\begin{proposition}
Let $X$ be a complex Banach space and $\Gamma$ a subset of $\mathcal{B}(X)$. 
Assume that for all $T$, $S\in\Gamma$ we have $TS\in\Gamma$ and let $(\lambda_T)_{T\in\Gamma}$ be a (multiplicative) semi-group inside $\mathbb{T}$. Then $\Gamma$ is recurrent if and only if 
$$\Gamma_1:=\{\lambda_T T \mbox{ : }T\in\Gamma\}$$
is recurrent
\end{proposition}
%%%%%%%%%%%%%%%%%%%%%%%%%%%%%%%%%%%%%%%%%%%%%%%%%%%%%%%%%%%%%%%%%%%%%%%%%%%%%%%%%%%%%%%%%%%%%%%%%%%%%%%%%%
\begin{proof}
By Theorem \ref{t2}, we have $Rec(\Gamma)=Rec(\Gamma_1)$ and then we use Theorem \ref{t1}.
\end{proof}
%%%%%%%%%%%%%%%%%%%%%%%%%%%%%%%%%%%%%%%%%%%%%%%%%%%%%%%%%%%%%%%%%%%%%%%%%%%%%%%%%%%%%%%%%%%%%%%%%%%%%%%%%%
\section{Recurrent $C$-Regularized Groups of Operators}

In this section, we study the particular case when $\Gamma$ is a $C$-regularized group of operators. 
%%%%%%%%%%%%%%%%%%%%%%%%%%%%%%%%%%%%%%%%%%%%%%%%%%%%%%%%%%%%%%%%%%%%%%%%%%%%%%%%%%%%%%%%%%%%%%%%%%%%%%%%%%
Recall from \cite{CKM}, that an entire $C$-regularized group is an operator family $(S(z))_{z\in\mathbb{C}}$ on $\mathcal{B}(X)$ that satisfies$:$
\begin{itemize}
\item[$(1)$] $S(0)=C;$
\item[$(2)$] $S(z+w)C = S(z)S(w)$ for every $z,$ $w\in\mathbb{C}$,
\item[$(3)$] The mapping $z \mapsto S(z)x$, with $z\in\mathbb{C}$, is entire for every $x \in X$.
\end{itemize}
%%%%%%%%%%%%%%%%%%%%%%%%%%%%%%%%%%%%%%%%%%%%%%%%%%%%%%%%%%%%%%%%%%%%%%%%%%%%%%%%%%%%%%%%%%%%%%%%%%%%%%%%%%

The next example shows that the recurrence of $C$-regularized groups of operators exists in each Banach space. 
%%%%%%%%%%%%%%%%%%%%%%%%%%%%%%%%%%%%%%%%%%%%%%%%%%%%%%%%%%%%%%%%%%%%%%%%%%%%%%%%%%%%%%%%%%%%%%%%%%%%%%%%%%
\begin{example}\label{example}
Let $X$ be a complex Banach space. For all $z\in\mathbb{C}$, let $S(z)$ be an operator defined on $X$ by
$$\begin{array}{ccccc}
S(z) & : & X & \longrightarrow & X \\
 & & x & \longmapsto & S(z)x=e^{z}x. \\
\end{array}$$
Let $U$ be a nonempty open subset of $X$. Then there exists $z\in\mathbb{C}$ such that
$$ S(z)(U)\cap U\neq \emptyset. $$
This means that $(S(z))_{z\in\mathbb{C}})$ is recurrent $C$-Regularized group of operators.
\end{example}
%%%%%%%%%%%%%%%%%%%%%%%%%%%%%%%%%%%%%%%%%%%%%%%%%%%%%%%%%%%%%%%%%%%%%%%%%%%%%%%%%%%%%%%%%%%%%%%%%%%%%%%%%
%%%%%%%%%%%%%%%%%%%%%%%%%%%%%%%%%%%%%%%%%%%%%%%%%%%%%%%%%%%%%%%%%%%%%%%%%%%%%%%%%%%%%%%%%%%%%%%%%%%%%%%%%
\begin{remark}
Let $X$ be a complex topological vector space and $(S(z))_{z\in\mathbb{C}})$ a $C$-regularized group of operators. The fact that $(S(z))_{z\in\mathbb{C}})$ is recurrent does not implies that $S(z_0)$ is recurrent for all $z_0\in\mathbb{C}$. Indeed, let $(S(z))_{z\in\mathbb{C}})$ be the $C$-Regularized group
 defined as in Example \ref{example}. Then $(S(z))_{z\in\mathbb{C}})$ is recurrent. 
 However, $Rec(S(z))=\emptyset$ whenever $\vert z\vert>1.$
\end{remark}
%%%%%%%%%%%%%%%%%%%%%%%%%%%%%%%%%%%%%%%%%%%%%%%%%%%%%%%%%%%%%%%%%%%%%%%%%%%%%%%%%%%%%%%%%%%%%%%%%%%%%%%%%%
%%%%%%%%%%%%%%%%%%%%%%%%%%%%%%%%%%%%%%%%%%%%%%%%%%%%%%%%%%%%%%%%%%%%%%%%%%%%%%%%%%%%%%%%%%%%%%%%%%%%%%%%%%
\begin{lemma}\label{lem}
Let $(S(z))_{z\in\mathbb{C}}$ be a recurrent $C$-regularized group on a complex Banach space.
Then $Cx\in Rec((S(z))_{z\in\mathbb{C}})$, for all $x\in Rec((S(z))_{z\in\mathbb{C}}).$ 
\end{lemma}
%%%%%%%%%%%%%%%%%%%%%%%%%%%%%%%%%%%%%%%%%%%%%%%%%%%%%%%%%%%%%%%%%%%%%%%%%%%%%%%%%%%%%%%%%%%%%%%%%%%%%%%%%%
\begin{proof}
Let $z\in \mathbb{C}$. By conditions $(1)$ and $(2)$ of the definition of  a $C$-regularized group we have
$$ S(z)C=S(0+z)C=S(0)S(z)=CS(z),$$
this means that $C$ commutes with every element of $(S(z))_{z\in\mathbb{C}}$. Hence, $C\in \{(S(z))_{z\in\mathbb{C}}\}^{'}$. By using Proposition \ref{p1} one can deduce 
$Cx\in Rec((S(z))_{z\in\mathbb{C}})$, for all $x\in Rec((S(z))_{z\in\mathbb{C}})$.
\end{proof}
%%%%%%%%%%%%%%%%%%%%%%%%%%%%%%%%%%%%%%%%%%%%%%%%%%%%%%%%%%%%%%%%%%%%%%%%%%%%%%%%%%%%%%%%%%%%%%%%%%%%%%%%%%
\begin{proposition}
Let $(S(z))_{z\in\mathbb{C}}$ be a recurrent $C$-regularized group on a complex Banach space.
If $C=I$ the identity operator on $X$, then $S(z)x\in Rec((S(z))_{z\in\mathbb{C}}))$ for all $x\in Rec((S(z))_{z\in\mathbb{C}})$ and for all $z\in\mathbb{C}$.
\end{proposition}
%%%%%%%%%%%%%%%%%%%%%%%%%%%%%%%%%%%%%%%%%%%%%%%%%%%%%%%%%%%%%%%%%%%%%%%%%%%%%%%%%%%%%%%%%%%%%%%%%%%%%%%%%%
\begin{proof}
By remarking that in this case we have $S(z)\in \{(S(z))_{z\in\mathbb{C}})\}^{'}$ and using Proposition \ref{p1}.
\end{proof}
%%%%%%%%%%%%%%%%%%%%%%%%%%%%%%%%%%%%%%%%%%%%%%%%%%%%%%%%%%%%%%%%%%%%%%%%%%%%%%%%%%%%%%%%%%%%%%%%%%%%%%%%%%
\begin{definition}
Let $(S(z))_{z\in\mathbb{C}}$ be a $C$-regularized group on a complex Banach space.
 Given another complex Banach space $X$ and an isomorphism $\phi$ from $Y$ onto $X$, the $C$-regularized group $(h(z))_{z\in\mathbb{C}}$ on Y , defining by
 $$h(z)=\phi^{-1}S(z)\phi$$
 is said to be similar to $(S(z))_{z\in\mathbb{C}}$.
\end{definition}
%%%%%%%%%%%%%%%%%%%%%%%%%%%%%%%%%%%%%%%%%%%%%%%%%%%%%%%%%%%%%%%%%%%%%%%%%%%%%%%%%%%%%%%%%%%%%%%%%%%%%%%%%%
\begin{proposition}
Let $(S(z))_{z\in\mathbb{C}}$ be a recurrent $C$-regularized group of operators on a complex Banach space $X$.
 If $(h(z))_{z\in\mathbb{C}}$ is a  $C$-regularized group of operators on a complex Banach space $Y$ similar to $(S(z))_{z\in\mathbb{C}}$, then
$(h(z))_{z\in\mathbb{C}}$ is recurrent on $Y$ . Moreover,
$$Rec((S(z))_{z\in\mathbb{C}})=\phi((h(z))_{z\in\mathbb{C}}).$$
\end{proposition}
%%%%%%%%%%%%%%%%%%%%%%%%%%%%%%%%%%%%%%%%%%%%%%%%%%%%%%%%%%%%%%%%%%%%%%%%%%%%%%%%%%%%%%%%%%%%%%%%%%%%%%%%%%
\begin{proof}
Direct consequence of Proposition \ref{p3}.
\end{proof}
%%%%%%%%%%%%%%%%%%%%%%%%%%%%%%%%%%%%%%%%%%%%%%%%%%%%%%%%%%%%%%%%%%%%%%%%%%%%%%%%%%%%%%%%%%%%%%%%%%%%%%%%%%%%%%%%%%%%%%%%%%%%%%%%%%%%%%%%%%%%%%%%%%%%%%%%%%%%%%%%%%%%%%%%%%%%%%%%%%%%%%%%%%%%%%%%%%%%%%%%%%%%%%%%%%%%


\begin{thebibliography}{9}
\bibliographystyle{alpha}
%%%%%%%%%%%%%%%%%%%%%%%%%%%%%%%%%%%%%%%%%%%%%%%%%%%%%%%%%%%%%%%%%%%%%%%%%%%%%%%%%%%%%%%%%%%%%%%%%%%%%%%%%%
\bibitem{somversupsetope}
Amouch M, Benchiheb O (2021) Some versions of supercyclicity for a set of operators. Filomat 35(5):1619-1627
%%%%%%%%%%%%%%%%%%%%%%%%%%%%%%%%%%%%%%%%%%%%%%%%%%%%%%%%%%%%%%%%%%%%%%%%%%%%%%%%%%%%%%%%%%%%%%%%%%%%%%%%%%
\bibitem{dissetope}
Amouch M, Benchiheb O (2020) Diskcyclicity of sets of operators and applications. Acta Math Sin Eng Ser. 36(11):1203-1220.
%%%%%%%%%%%%%%%%%%%%%%%%%%%%%%%%%%%%%%%%%%%%%%%%%%%%%%%%%%%%%%%%%%%%%%%%%%%%%%%%%%%%%%%%%%%%%%%%%%%%%%%%%%
\bibitem{cycsetope}
Amouch M, Benchiheb O (2019) On cyclic sets of operators." Rendiconti del Circolo Matematico di Palermo Series 2. 68(3):521-529
%
\bibitem{lindynsetope}
Amouch M, Benchiheb O (2019) On linear dynamics of sets of operators.
Turk J Math
43:402-411
\bibitem {AKH}  Ansari M,  Hedayatian K, Khani-robati, B (2018)
On the density and transitivity of sets of operators.
Turk J Math 42(1):181-189 
%%%%%%%%%%%%%%%%%%%%%%%%%%%%%%%%%%%%%%%%%%%%%%%%%%%%%%%%%%%%%%%%%%%%%%%%%%%%%%%%%%%%%%%%%%%%%%%%%%%%%%%%%%
%%%%%%%%%%%%%%%%%%%%%%%%%%%%%%%%%%%%%%%%%%%%%%%%%%%%%%%%%%%%%%%%%%%%%%%%%%%%%%%%%%%%%%%%%%%%%%%%%%%%%%%%%%
\bibitem {AKH1} Ansari M.=, Hedayatian K, Khani Robati B., Moradi A (2018) A note on topological and strict transitivity. Iran J Sci Technol Trans Sci 42(1):59-64
%%%%%%%%%%%%%%%%%%%%%%%%%%%%%%%%%%%%%%%%%%%%%%%%%%%%%%%%%%%%%%%%%%%%%%%%%%%%%%%%%%%%%%%%%%%%%%%%%%%%%%%%%%%
\bibitem{Bayart} Bayart F,  Matheron E (2009) Dynamics of linear operators.
New York, NY, USA, Cambridge University Press 
%%%%%%%%%%%%%%%%%%%%%%%%%%%%%%%%%%%%%%%%%%%%%%%%%%%%%%%%%%%%%%%%%%%%%%%%%%%%%%%%%%%%%%%%%%%%%%%%%%%%%%%%%%%%
\bibitem{CKM}  Conejero JA, Kostic M, Miana PJ, Murillo-Arcila M (2016) Distributionally chaotic families of operators on Frechet spaces.
Commun Pure Appl Anal 15(5):1915-1939
%%%%%%%%%%%%%%%%%%%%%%%%%%%%%%%%%%%%%%%%%%%%%%%%%%%%%%%%%%%%%%%%%%%%%%%%%%%%%%%%%%%%%%%%%%%%%%%%%%%%%%%%%%%%
\bibitem{CMP}
Costakis G, Manoussos A, Parissis I (2014) Recurrent linear operators.
Complex Anal Oper Th
8:1601-1643
%%%%%%%%%%%%%%%%%%%%%%%%%%%%%%%%%%%%%%%%%%%%%%%%%%%%%%%%%%%%%%%%%%%%%%%%%%%%%%%%%%%%%%%%%%%%%%%%%%%%%%%%%%%%%%%%%%%
\bibitem{CP} Costakis G, Parissis I (2012) Szemer\'{e}di's theorem, frequent hypercyclicity and multiple recurrence.
Math Scand
110: 251-272
%%%%%%%%%%%%%%%%%%%%%%%%%%%%%%%%%%%%%%%%%%%%%%%%%%%%%%%%%%%%%%%%%%%%%%%%%%%%%%%%%%%%%%%%%%%%%%%%%%%%%%%%%%%%
\bibitem{Furstenberg} Furstenberg H (1981) Recurrence in ergodic theory and combinatorial number theory. Princeton: Princeton
University Press, M. B. Porter Lectures 
%%%%%%%%%%%%%%%%%%%%%%%%%%%%%%%%%%%%%%%%%%%%%%%%%%%%%%%%%%%%%%%%%%%%%%%%%%%%%%%%%%%%%%%%%%%%%%%%%%%%%%%%%%%%
\bibitem{Peris}  Grosse-Erdmann K.-G, Peris A (2011) Linear Chaos.
(Universitext). Springer, London 
%%%%%%%%%%%%%%%%%%%%%%%%%%%%%%%%%%%%%%%%%%%%%%%%%%%%%%%%%%%%%%%%%%%%%%%%%%%%%%%%%%%%%%%%%%%%%%%%%%%%%%%%%%%%
\bibitem{HW}  Hilden HM,  Wallen LJ (1994) Some cyclic and non-cyclic vectors of certain operators.
Indiana Univ Math J
23:557-565
%%%%%%%%%%%%%%%%%%%%%%%%%%%%%%%%%%%%%%%%%%%%%%%%%%%%%%%%%%%%%%%%%%%%%%%%%%%%%%%%%%%%%%%%%%%%%%%%%%%%%%%%%%%%%%
\bibitem{KBA}Karim N, Benchiheb O,  Amouch M (2022) Recurrence of multiples of composition operators on weighted Dirichlet spaces.
Adv Oper Theory 7(23) doi.org/10.1007/s43036-022-00186-1
\end{thebibliography}
\end{document}